\documentclass[12 pt,reqno]{article}
\usepackage{amssymb, amsmath}
\usepackage{amssymb,amsthm}
\usepackage{rotating}
\newenvironment{rcases}
{\left.\begin{aligned}}
	{\end{aligned}\right\rbrace}
\newtheorem{Theorem}{Theorem}
\newtheorem{Remark}{Remark}
\newtheorem{Lemma}{Lemma}
\newtheorem{Proposition}{Proposition}

\newtheorem{Definition}{Definition}
\numberwithin{Definition}{section}
\numberwithin{Theorem}{section}
\numberwithin{Lemma}{section}
\numberwithin{Proposition}{section}

\newtheorem{Example}{Example}
\numberwithin{Corollary}{section}
\numberwithin{Example}{section}
\numberwithin{Remark}{section}
\begin{document}
\title{ On an Exact Penalty Result and New Constraint Qualifications for Mathematical Programs with Vanishing Constraints}	
	\author{Triloki Nath\thanks{Department of Mathematics and Statistics, Dr. Harisingh Gour Vishwavidyalaya, Sagar, Madhya Pradesh-470003, INDIA, Email- tnverma07@gmail.com}  and  Abeka Khare\thanks{Department of Mathematics and Statistics, Dr. Harisingh Gour Vishwavidyalaya, Sagar, Madhya Pradesh-470003, INDIA, Email- abekakhare2012@gmail.com} }	
\maketitle
\date{}
\begin{abstract}
 In this paper,  the mathematical programs with vanishing constraints or MPVC are considered. We prove that an MPVC-tailored penalty function, introduced in \cite{hoheisel}, is still  exact under a very weak and new constraint qualification. Most importantly, this constraint qualification is shown to be strictly stronger than MPVC-Abadie constraint qualification.
\end{abstract}
		\section{Introduction}
		In this paper, we consider \emph{mathematical program with vanishing constraints} (or MPVC in short),                                                                    having the following mathematical form:
		\begin{eqnarray}
\label{initprblm}
		\nonumber {\rm min}_{x \in \mathbb{R}^n} ~~ f(x)\\
	 \nonumber {\rm s.t.}~~ g_i(x) &\leqslant&  0 ~~~\forall~  i=1,2,...,m,~~\\
		\nonumber h_j(x) &=&  0 ~~~\forall~  j=1,2,...,l\\
		\nonumber 	H_i(x) &\geqslant&  0 ~~~\forall~  i=1,2,...,q\\
		G_i(x) H_i(x) &\leqslant&  0 ~~~\forall~  i=1,2,...,q.
		\end{eqnarray}
		where all functions $f : \mathbb{R}^n \rightarrow \mathbb{R}, ~g_i : \mathbb{R}^n \rightarrow \mathbb{R},~h_i : \mathbb{R}^n \rightarrow \mathbb{R},~G_i : \mathbb{R}^n \rightarrow \mathbb{R},~H_i~:~\mathbb{R}^n~\rightarrow~\mathbb{R}$ are assumed to be continuously differentiable. The nomenclature is justified because its implicit sign constraint function $G_i(x) \leqslant 0$ vanishes whenever $H_i(x) = 0$. We assume $\mathcal{C}$ as the feasible region for this MPVC throughout the paper.\\
\hspace*{0.5 cm} The MPVC plays very important roles in many fields, such as truss topology optimization \cite{achtziger} and robot motion planning\cite{Latombe,Kirches2}. The constrained optimization problems arising in applied sciences, engineering and economics seek the algorithms, which rely on standard Karush-Kuhn-Tucker (KKT) conditions. The major difficulty in solving MPVC is that it typically violates most of the standard constraint qualifications (CQs), and hence the standard  KKT conditions are not relevant in MPVC context. \\
		\hspace*{0.5 cm} It is known that MPVC is closely related to the well known MPEC (\emph{mathematical programs with equilibrium constraints}), and this leads to an analogous development for MPVC. In literature, a lot of research has been carried out for  MPVC regarding its stationarity conditions and constraint qualifications, see e.g.\cite{achtziger, hoheisel, hoheisel 2, hoheisel 4, hoheisel 3, q hu}, and for the algorithmic aspects we refer to \cite{achtziger2,hoheisel 1,Kanzow 2}. The exact penlty results are also associated with some sort of constraint qualifications. But, in this direction a very few work has been appeared, namely \cite{hoheisel, q hu}. To the best of our knowledge, \cite[Theorem 4.5]{hoheisel} is the first exact penalty result under MPVC-MFCQ for the following MPVC-tailored penalty function
\begin{equation*}
		P_\alpha (x) ~=~ f(x) + \alpha [ \sum_{i=1}^{m} \max \{0,g_i(x)\} +  \sum_{j=1}^{l} |h_j(x)| + \sum_{i=1}^{q}  \max \{0,-H_i(x), \min \{G_i(x), H_i(x)\} \}]
		\end{equation*}
\hspace*{0.5 cm} In \cite[corollary 6.8]{hoheisel}, the authors also discussed exact penalization of classical $l_1-$ penalty function associated to MPVC ($g$ and $h$ absent), given as follows
	\begin{equation*}
	P^1_\alpha(x) ~=~ f(x) + \alpha \sum_{i=1}^{q} \max \{-H_i(x), 0 \} + \alpha \sum_{i=1}^{q} \max \{G_i(x) H_i(x), 0\}.
	\end{equation*}
	\hspace*{0.5 cm} The authors concluded that exactness condition for MPVC-tailored-penalty function, namely MPVC-MFCQ, does not guarantee the exactness of $l_1$-penalty function and found MPVC-LICQ to be a sufficient condition for exactness of $l_1$-penalty function, but, under a very strong assumption that biactive set $I_{00}$ is empty. One can see that under this restriction an MPVC becomes, locally, a standard nonlinear program and loses its challenging combinatorial structure to some degree, see \cite{izmailov}. Later, Hu improved this result with MPVC-\emph{generalized pseudonormality CQ} in \cite[Theorem 3.2]{q hu}, which works under an assumption that includes the non-emptyness of biactive set. In future some better  results  regarding the  exactness of this  $l_1$- penalty function can also be concluded by imposing some relaxed assumptions than  \cite{hoheisel,q hu}.  It is still an open question, if we do not impose any condition on bi-active set . \\\\
\hspace*{0.5 cm}  Following the above discussion, one may naturally ask for conditions, weaker than MPVC-MFCQ, under which exact penalty result holds, atleast for $P_\alpha (x)$ the specialized one.\\

\hspace*{0.5 cm} The goal of this paper is bipartite, first we answer affirmatively, in a better way, that MPVC-tailored-penalty function still remains exact at any local minimizer under the MPVC- \emph{generalized quasinormality}, which is much weaker than  MPVC-MFCQ. The significance of our result will be illustrated in section 3 with an example. Secondly, we derive  relationships among some important old and new CQs of MPVC, defined so far. It is known \cite[Theorem 3.4]{hoheisel 2} that MPVC-GCQ (G-\emph{Guignard}) is the weakest CQ under which M-\emph{stationarty} conditions holds for MPVC. The MPVC-ACQ (A-\emph{Abadie}) is easily tractable and strictly stronger than MPVC-GCQ. In what follows, sufficient conditions have been investigated for MPVC-ACQ, see \cite{hoheisel 4,hoheisel 2}. We prove that MPVC-generalized quasinormality implies MPVC-ACQ in Theorem \ref{relation}. Although, implications among some stronger constraint qualification has been already established, see  \cite{hoheisel, q hu}. We provide examples to illustrate that relationships are strict among them.\\
\hspace*{0.5 cm} The rest of the paper is organized as follows. Section \ref{Pre} contains some background materials required to understand the present work. In section \ref{Expen} we derive sufficient condition for MPVC-tailored penalty function to be exact. The section \ref{relcq} is devoted to establish the relationship among the constraint qualifications of MPVC, and we finish with some concluding remarks in section \ref{Conc}.
	
		\section{Preliminaries}
\label{Pre}
		Here, we adopt the following notations for index sets from \cite{hoheisel 2} for an arbitrary feasible point $x^\ast$.
		\begin{eqnarray*}
			I_g&:=& \{i ~|~ g_i(x^\ast) = 0 \}\\
			I_+&:=& \{i ~|~ H_i(x^\ast) > 0 \},~~~~I_0 := \{i ~|~ H_i(x^\ast) = 0 \}\\
			I_{+0}&:=& \{i~|~H_i(x^\ast) > 0 ~,~ G_i(x^\ast) = 0\},~~~I_{+-}:= \{i~|~H_i(x^\ast) > 0 ~~ G_i(x^\ast) < 0\}\\
			I_{0+} &:=& \{i~|~H_i(x^\ast) = 0 ~,~ G_i(x^\ast) > 0\},~~~
			I_{0-}:= \{i~|~H_i(x^\ast) = 0 ~,~ G_i(x^\ast) < 0\}\\
I_{00}&:=& \{i~|~H_i(x^\ast) = 0 ~,~ G_i(x^\ast) = 0\}.
		\end{eqnarray*}
		Next,  we recall concepts of well defined cones from non smooth analysis \cite{Rocka}.
		\begin{Definition}
			{\rm (1) Let $C \subset \mathbb{R}^n$ be a nonempty closed set and $x^\ast \in C$. The (\emph{Bouligand})  tangent cone (or  \emph{contingent} cone) of $C$ at $x^\ast$ is defined as
			\begin{eqnarray*}
				T_C(x^\ast) &:=& \{ d \in \mathbb{R}^n ~|~ \exists \{x^k\} \rightarrow_C x^\ast, ~\{t_k\} \downarrow 0 ~:~\frac{x^k - x^\ast}{t_k} \rightarrow d \} \\
				&:=& \{ d \in \mathbb{R}^n ~|~ \exists \{d^k\} \rightarrow d, ~\{t_k\} \downarrow 0 ~:~x^\ast + t_k d^k \in C ~~\forall~k \in \mathbb{N} \},
			\end{eqnarray*}
			where $\{x^k\} \rightarrow_C x^\ast$ denotes a sequence $\{x^k\}$ converging to $x^\ast$ and satisfying $x^k \in C ~\forall~ k \in \mathbb{N}$.\\
			The vector $d \in T_{C}(x^\ast)$ is called a tangent vector to ${C}$ at $x^\ast$. \\\\
		{\rm (2)}  Let $C \subset \mathbb{R}^n$ be a nonempty closed set and $x^\ast \in C$. The  \emph{Fr\'echet} normal cone of $C$ at $x^\ast$ is defined as
			\begin{equation*}
			N_C^F(x^\ast)~:=~T_C(x^\ast)^\circ
			\end{equation*}
			{\rm (3)} Let $C \subset \mathbb{R}^n$ be a nonempty closed set and $x^\ast \in C$. The \emph{limiting} normal cone of $C$ at $x^\ast$ is defined as	
			\begin{equation*}
			N_C(x^\ast) ~:=~\{d \in \mathbb{R}^n ~|~\exists \{x^k\} \rightarrow_C x^\ast, d^k \in N^F_C(x^k)~:~ d^k \rightarrow d \}.
			\end{equation*}}
		\end{Definition}
The graph of the  multifunction
$\Phi : \mathbb{R}^n \rightrightarrows \mathbb{R}^m$ is defined as $gph \Phi:=  \{ (x,y)~|y \in \Phi(x)\}$. For $x \in \mathbb R^n$ and $\delta>0$, the set $\mathbb B(x, \delta):= \{ y \in \mathbb R^{n}|~\|y-x \| < \delta \}$ is open ball. Without loss of generality, the $\|.\|$ will be taken as $l_1$-norm.\\

		Now, we discuss some well known constraint qualifications of nonlinear programming in the context of MPVC.
\begin{Definition} {\rm \cite{hoheisel 2}
	\label{Def1} A vector $x^\ast \in \mathcal{C}$ is said to satisfy MPVC-\emph{linearly independent constraint qualification}  (or MPVC-LICQ) if the gradients \\
	$~~~~~~~~~~~\{\nabla g_i(x^\ast) | i \in I_g\} ~\cup ~ \{\nabla h_i(x^\ast) | i = 1,...,p\} ~\cup ~ \{\nabla G_i(x^\ast) | i \in I_{+0}~\cup~ I_{00}\} \\~~~~~~~~~~~ \cup~ \{\nabla H_i(x^\ast) | i \in I_0\} $ \\
	are linearly independent.}
	\end{Definition}
		\begin{Definition}
			{\rm \cite{hoheisel} A vector $x^\ast \in \mathcal{C}$ for (MPVC) satisfies MPVC-\emph{Mangasarian Fromovitz constraint qualification} (or MPVC-MFCQ) if the \\
			$~~~~~~~~~~~~~~\nabla h_i(x^\ast) ~~(i= 1, ..., p),~~~~\nabla H_i(x^\ast)~~(i \in I_{0+} \cup I_{00})$\\
			are linearly independent and there exist a vector $d \in \mathbb{R}^n$ such that \\
			$~~~~~~~~\nabla h_i(x^\ast)d =0 ~~~(i=1,...,p), ~~~~\nabla H_i (x^\ast)^T d = 0 ~~~(i \in I_{0+} \cup I_{00}) \\
			~~~~~~~\nabla g_i (x^\ast)^T d < 0 ~~~(i \in I_g), ~~~~\nabla H_i (x^\ast)^T d > 0 ~~~(i \in I_{0-}),\\
			~~~~~~~	\nabla G_i(x^\ast)^T d < 0 ~~~(i \in I_{+0} \cup I_{00}). $}
		\end{Definition}
		In the spirit of MPEC-GMFCQ  \cite{Kanzow 2}, the following MPVC-GMFCQ is defined as follows.
		\begin{Definition}
			\label{gmfcq 1}
			{\rm \cite{q hu}  A vector $x^\ast \in \mathcal{C}$ is said to satisfy \textit{MPVC-generalized MFCQ} (MPVC-GMFCQ) if there is no multiplier $(\lambda, \mu, \eta^H, \eta^G)~\neq~0$ such that \\
				
				$ {\rm (i)}~\sum_{i=1}^{m} \lambda_i \nabla g_i(x^\ast) + \sum_{i=1}^{l} \mu_i \nabla h_i (x^\ast) + \sum_{i=1}^{q} \eta^G_i \nabla G_i(x^\ast)
				-\sum_{i=1}^{q} \eta^H_i \nabla H_i(x^\ast)~=~0 $\\
				
				$ {\rm (ii)} ~~~~\lambda_i \geqslant 0 ~~~\forall~i \in I_g,~~~\lambda_i = 0 ~~~\forall~i \notin I_g$
				\begin{eqnarray*}
					{\rm and}~~~\eta_i^G &=& 0 ~~~~\forall~i \in I_{+-} \cup I_{0-} \cup I_{0+},~~\eta_i^G  \geqslant 0~\forall i \in I_{+0} \cup I_{00}\\
					\eta_i^H &=& 0 ~~~~\forall~i \in I_+,~~\eta_i^H  \geqslant  0~\forall ~i \in I_{0-}~~ {\rm and}~~\eta_i^H ~ {\rm is}~ {\rm free}~ \forall~ i \in I_{0+}\\
					\eta_i^H  \eta_i^G &=& 0~~~~\forall~ i \in I_{00}
				\end{eqnarray*}}
		\end{Definition}
		\begin{Definition}
			\label{acq}
			{\rm \cite{hoheisel 1}}
		MPVC-Abadie CQ (or MPVC-ACQ) holds at $x^\ast \in \mathcal{C}$, if
		\begin{equation*}
		T_\mathcal{C} (x^\ast) ~=~ L_{MPVC} (x^\ast)
		\end{equation*}
		where $L_{MPVC}(x^\ast)$ is the MPVC-linearized tangent cone and defined as {\rm \cite[Lemma 3.2.1]{hoheisel 1}}
				\begin{eqnarray*}
					L_{MPVC}(x^\ast)	&=& \{ d \in \mathbb{R}^n ~|~ \nabla g_i(x^\ast)^T d \leq0 ~~~~\forall ~i \in I_g \\
					&&~~~~~~~~~~~~~~ \nabla h_i(x^\ast)^T d = 0 ~~~~\forall ~i = 1,...,p \\
					&&~~~~~~~~~~~~~~ \nabla H_i(x^\ast)^T d = 0 ~~~~\forall ~i \in I_{0+} \\
					&&~~~~~~~~~~~~~~ \nabla H_i(x^\ast)^T d \geq 0 ~~~~\forall ~i \in I_{00} \cup I_{0-} \\
					&&~~~~~~~~~~~~~~ \nabla G_i(x^\ast)^T d \leq 0 ~~~~\forall ~i \in I_{+0} \}
				\end{eqnarray*}
			\end{Definition}
		\begin{Definition}
			{\rm \cite{q hu}}  A vector $x^\ast \in \mathcal{C}$ is said to satisfy \textit{MPVC-generalized pseudonormality}, if there is no multiplier $(\lambda, \mu, \eta^H, \eta^G)~\neq~0$ such that \\
			
			$  {\rm (i)} ~\sum_{i=1}^{m} \lambda_i \nabla g_i(x^\ast) + \sum_{i=1}^{l} \mu_i \nabla h_i (x^\ast) + \sum_{i=1}^{q} \eta^G_i \nabla G_i(x^\ast)
			-\sum_{i=1}^{q} \eta^H_i \nabla H_i(x^\ast)~=~0 $\\
			
			$ {\rm (ii)} ~~~~\lambda_i \geqslant 0 ~~~\forall~i \in I_g,~~~\lambda_i = 0 ~~~\forall~i \notin I_g$
				\begin{eqnarray*}
					{\rm and}~~~\eta_i^G &=& 0 ~~~~\forall~i \in I_{+-} \cup I_{0-} \cup I_{0+},~~\eta_i^G  \geqslant 0~\forall i \in I_{+0} \cup I_{00}\\
					\eta_i^H &=& 0 ~~~~\forall~i \in I_+,~~\eta_i^H  \geqslant  0~\forall ~i \in I_{0-}~~ {\rm and}~~\eta_i^H ~ {\rm is}~ {\rm free}~ \forall~ i \in I_{0+}\\
					\eta_i^H  \eta_i^G &=& 0~~~~\forall~ i \in I_{00}
				\end{eqnarray*}
			
			$ {\rm (iii)} $ there is a sequence $\{x^k\} \rightarrow x^\ast$ such that the following is true for all $k \in \mathbb{N}$ \\
			\begin{equation*}
			\sum_{i=1}^{m} \lambda_i g_i(x^k)~+~\sum_{i=1}^{p} \mu_i h_i(x^k)~+~\sum_{i=1}^{q} \eta_i^G G_i(x^k)~-~\sum_{i=1}^{q} \eta_i^H H_i(x^k)~>~0.
			\end{equation*}
		\end{Definition}
			\begin{Definition}
			 A vector $x^\ast \in \mathcal{C}$ is said to satisfy \textit{MPVC-generalized quasinormality}, if there is no multiplier $(\lambda, \mu, \eta^H, \eta^G)~\neq~0$ such that \\
			
			 $ {\rm (i)} \sum_{i=1}^{m} \lambda_i \nabla g_i(x^\ast) + \sum_{i=1}^{l} \mu_i \nabla h_i (x^\ast) + \sum_{i=1}^{q} \eta^G_i \nabla G_i(x^\ast)
			 -\sum_{i=1}^{q} \eta^H_i \nabla H_i(x^\ast)~=~0 $\\
			
			 $ {\rm (ii)} ~~~~\lambda_i \geqslant 0 ~~~\forall~i \in I_g,~~~\lambda_i = 0 ~~~\forall~i \notin I_g$
				\begin{eqnarray*}
					{\rm and}~~~\eta_i^G &=& 0 ~~~~\forall~i \in I_{+-} \cup I_{0-} \cup I_{0+},~~\eta_i^G  \geqslant 0~\forall i \in I_{+0} \cup I_{00}\\
					\eta_i^H &=& 0 ~~~~\forall~i \in I_+,~~\eta_i^H  \geqslant  0~\forall ~i \in I_{0-}~~ {\rm and}~~\eta_i^H ~ {\rm is}~ {\rm free}~ \forall~ i \in I_{0+}\\
					\eta_i^H  \eta_i^G &=& 0~~~~\forall~ i \in I_{00}
				\end{eqnarray*}					
			 $ {\rm (iii)} $ There is a sequence $\{x^k\} \rightarrow x^\ast$ such that the following is true $\forall k \in \mathbb{N} $, we have
			 \begin{eqnarray*}
			 	\lambda_i &>&0  ~~\Rightarrow ~~\lambda_i g_i(x^k) >  0 ~~~ \{i=1,...,m\}\\
			 	\mu_i &\neq&0  ~~\Rightarrow ~~\mu_i h_i(x^k) >  0 ~~~ \{i=1,...,p\}\\
			 	\eta^H_i &\neq&0  ~~\Rightarrow ~~\eta^H_i H_i(x^k) <  0 ~~~ \{i=1,...,q\}\\
			 	\eta^G_i &>&0  ~~\Rightarrow ~~\eta^G_i G_i(x^k) >  0 ~~~ \{i=1,...,q\}	
			 \end{eqnarray*}	
			\end{Definition}
		we have following relationships in these CQ as shown in \cite[Proposition 2.1]{q hu} and further implication  in \cite{khare}.
		\begin{Proposition}\label{implcq}
			MPVC-LICQ $\Rightarrow$ MPVC-MFCQ $\Rightarrow$	MPVC-GMFCQ $\Rightarrow$ MPVC-generalized pseudonormality $\Rightarrow$ MPVC-generalized quasinormality.
\end{Proposition}
\begin{Remark}
The implications in Proposition \ref{implcq} are strict. First and last implications are obviously strict. We illustrate in the following examples that MPVC-GMFCQ is srtrictly weaker than MPVC-MFCQ and MPVC-generalized pseudonormality is strictly weaker than MPVC-GMFCQ.
\end{Remark}
\begin{Example}
\label{exmpl1}consider the following MPVC
\begin{eqnarray*}
		\min ~f(x)~~~~ &&\\
		g(x)= x_1 - x_2 &\leqslant&0 \\
			H(x) = x_1 &\geqslant& 0 \\
			G(x) H(x) = x_1 x_2 &\leqslant&0
		\end{eqnarray*}
	here $x^\ast = (0, 0)$ is feasible point and all constraints are active at $(0,0)$. At $x^\ast = (0,0)$, MPVC-MFCQ does not hold: since
$~\nabla H(x^\ast)= \left(
                        \begin{array}{c}
                          1 \\
                          0 \\
                        \end{array}
                      \right)$   is linearly independent\\
			 and if there exist a vector $d = (d_1,d_2)^T\in \mathbb{R}^2$ such that
			\begin{eqnarray*}
			\nabla g (x^\ast)^T d&:=& \left(\begin{array}{cc}1 & -1 \end{array}\right)
			 \left( \begin{array}{c}
                          d_1 \\
                          d_2 \\
                        \end{array}
                      \right) < 0\\
			\nabla H (x^\ast)^T d &:=& \left(\begin{array}{cc}1 & 0 \end{array}\right)
			 \left( \begin{array}{c}
                          d_1 \\
                          d_2 \\
                        \end{array}
                      \right)= 0 \\
				\nabla G(x^\ast)^T d &:= & \left(\begin{array}{cc}0 & 1\end{array}\right)
			 \left( \begin{array}{c}
                          d_1 \\
                          d_2 \\
                        \end{array}
                      \right)< 0.
                      \end{eqnarray*}
Then $d_2 \geqslant 0$ and $d_2<0$ both hold, which is a contradiction. Hence, MPVC-MFCQ does not hold.
 But, by  definition MPVC-GMFCQ obviously holds. For, suppose
     \begin{eqnarray*}
     \lambda  \left( \begin{array}{c}
                           1 \\
                          -1\\
                        \end{array}
                      \right) + \eta^{G} \left( \begin{array}{c}
                          0 \\
                          1 \\
                        \end{array}
                      \right)-\eta^{H} \left( \begin{array}{c}
                          1 \\
                          0 \\
                        \end{array}
                      \right) =  \left( \begin{array}{c}
                          0 \\
                          0 \\
                        \end{array}
                      \right)
     \end{eqnarray*}
     with restrictions $\lambda \geqslant 0, ~\eta^{G}\geqslant 0$ and $ \eta^{H} \eta^{G} =0$. Then, we have $\lambda =\eta^{G} = \eta^{H}=0$.
\end{Example}
We have another example to illustrate that MPVC-generalized pseudonormality is strictly weaker than MPVC-GMFCQ.
\begin{Example}
\label{exmpl2}
Consider the typical MPVC problem in $\mathbb R^2$
	\begin{eqnarray*}
	\min ~x_1^2 + x_2^2 &&\\
	g(x) = x_1 &\leq& 0 \\
	H(x) = x_2 &\geq& 0\\
	G(x) H(x) = -x_1 x_2 &\leq& 0.	
	\end{eqnarray*}
	Then $x^\ast = (0,0)$ is a feasible point and and all constraints are active at $x^\ast$. To prove that  MPVC-GMFCQ fails to hold at $x^\ast$, we need to find $(\lambda, \eta^{G}, \eta^{H}) \neq 0$  such that
     \begin{eqnarray*}
     \lambda  \left( \begin{array}{c}
                           1 \\
                          0\\
                        \end{array}
                      \right) + \eta^{G} \left( \begin{array}{c}
                          -1 \\
                          0 \\
                        \end{array}
                      \right)-\eta^{H} \left( \begin{array}{c}
                          0 \\
                          1 \\
                        \end{array}
                     \right) =  \left( \begin{array}{c}
                          0 \\
                          0 \\
                        \end{array}
                      \right)
     \end{eqnarray*}
     with restrictions $\lambda \geqslant 0, ~\eta^{G}\geqslant 0$ and $ \eta^{H} \eta^{G} =0$. Then, clearly, all the multipliers with above properties can be taken as    $(\lambda, \eta^{G}, \eta^{H}) =c (1, 1, 0)$ with $c>0$. Thus, MPVC-GMFCQ is violated at $x^\ast$.\\
	
	 On the other hand
	 \begin{eqnarray*}
	 \lambda x_1^k + \eta^G(-x_1^k)-\eta^H x_2^k = c x_1^k - cx_1^k-0=0
	 \end{eqnarray*}
	 holds for all sequences $\{x^k\} \rightarrow x^\ast$. Hence, MPVC-generalized pseudonormality holds.
	 \end{Example}
		\section{An Exact Penalty Result for MPVC}
\label{Expen}
	Here,  we provide the exactness result for MPVC-tailored penalty function introduced in \cite[equation (26)]{hoheisel} under MPVC-generalized quasinormality, which is much weaker than MPVC-MFCQ.
In order to derive exact penalty function, we rewrite the MPVC first in  vector form as :
\begin{equation}
\label{mpvc vector}
\min ~ f(x)	~~~~s.t.~~F(x) ~\in ~ \Delta,
\end{equation}
where
\begin{equation*}
F(x)~:= ~\left(\begin{array}{cccc}
g_i(x)_{i=1,...,m}\\h_i(x)_{i=1,...,l}\\ \left( \begin{array}{cc}
G_i(x)\\H_i(x)
\end{array}\right)_{i=1,...,q}
\end{array} \right)
\end{equation*}
and
\begin{equation*}
\Delta ~ :=~ \left( \begin{array}{ccc}
(-\infty, 0]^m \\ \{0\}^l \\ \Omega^q
\end{array}\right)
\end{equation*}
with
\begin{equation*}
\Omega ~:=~\{(a,b) \in \mathbb{R}^2 ~|~b \geq 0,~ ab \leq 0  \}
\end{equation*}
 Since we are studying exactness of MPVC-tailored penalized problem, so we have to write first a penalty function  associated with (\ref{mpvc vector}) as (see \cite{hoheisel})
\begin{equation}
\label{penalty}
	P_\alpha(x) ~:=~ f(x) + \alpha ~{\rm dist}_\Delta (F(x))
\end{equation}
or
\begin{equation*}
P_\alpha(x) ~:=~ f(x) + \alpha~\left[ \sum_{i=1}^{m}{\rm dist}_{(-\infty, 0]} (g_i(x)) + \sum_{j=1}^{l} {\rm dist}_{\{0\}} (h_j(x)) + \sum_{i=1}^{q} dist_\Omega (G_l(x), H_l(x))\right]	
\end{equation*}
\begin{equation}
P_\alpha(x) ~:=~ f(x) + \alpha~ \left(  ||g^+(x)||_1 + ||h(x)||_1 + \sum_{i=1}^{q} dist_\Omega (G_l(x), H_l(x))\right)
\end{equation}
 where ${\rm dist}_S(x)$ is the distance in $l_1$-norm from $x$ to set $S$  and $g^+(x)=\max \{0, g(x)\}$, here max function $g^+$ is defined cmponentwise. Further, by using distance function for vanishing constraint \cite[Lemma 4.6]{hoheisel}, we have
\begin{equation*}
	P_\alpha (x) = f(x) + \alpha \left [ \sum_{i=1}^{m} | g_i^+(x)| + \sum_{j=1}^{l} |h_j(x)| + \sum_{i=1}^{q} \max \{ 0, -H_i(x), \min\{G_i(x), H_i(x)\} \} \right ]
\end{equation*}
\hspace*{0.5 cm} In order to derive exact penalty condition, we need some extra results. Here we have such result from \cite[Theorem 5.2]{khare} which states about the local error bound property of MPVC at a feasible point.
\begin{Lemma}
	\label{error bound}
	Let $x^\ast \in \mathcal{C}$ the feasible region of MPVC. If $x^\ast$ is MPVC-generalized quasinormal, then there are $\delta, c > 0$ such that
	\begin{equation}
	\label{error bound lemma}
	dist_\mathcal{C}(x) \leqslant c \left( ||h(x)||_1 + ||g^+(x)||_1 + \sum_{i=1}^{q} dist_\Omega (G_l(x), H_l(x)) \right)
	\end{equation}
	holds for all $x \in \mathbb{B}(x^\ast, \delta / 2)$.
	\label{thm last}
\end{Lemma}
With the help of above Lemma we can conclude the main result of this section.
\begin{Theorem}
\label{thm-extpenl}
Let $x^\ast $ be a local minimizer of MPVC  with $f$ locally Lipschitz at $x^\ast$ with Lipschitz constant $L > 0$. If MPVC-generalized-quasinormality holds at $x^\ast$, then the penality function $P_\alpha$ defined in (\ref{penalty}) is exact at $x^\ast$.
\end{Theorem}
\begin{proof}
	We have local error bound property for smooth MPVC, we redefine the constants $\delta ~{\rm and}~c$ in Lemma  \ref{error bound}, then (\ref{error bound lemma}) can be expressed as follows
	\begin{equation*}
		{\rm dist}_\mathcal{C}(x) \leq c~ {\rm dist}_\Delta (F(x))
	\end{equation*}
	for all $x \in \mathbb{B}(x^\ast, \delta)$. Now choose $\epsilon > 0$ such that $2 \epsilon < \delta$ and $f$ achieves global minimum at $x^\ast$ on $\mathbb{B}(x^\ast, 2 \epsilon) \cap \mathcal{C}$. Since $f$ is locally Lipschitz at $x^\ast$, we can assume, without loss of generality, that $L$ is the Lipschitz constant of $f$ in $\mathbb{B} (x^\ast, 2 \epsilon)$. Then following holds for all $x$  in $ \mathbb{B} (x^\ast, \epsilon)$ :\\
	Choose $x^\pi \in \Pi _\mathcal{C}(x)~=~\{z \in \mathcal{C} ~|~{\rm dist}_{\mathcal{C}}(x) = ||z-c||_1 \}$ arbitrarily, that is, $\Pi_\mathcal{C}(x)$ is the projections of $x$ onto $\mathcal{C}$. Then
	\begin{equation*}
	||x^\pi - x||_1 \leq ||x^\ast - x||_1 \leq \epsilon ~\Rightarrow~||x^\pi - x^\ast||_1 \leq ||x^\pi - x||_1 + ||x-x^\ast||_1 \leq 2 \epsilon
	\end{equation*}
	and consequently, we have
	\begin{eqnarray*}
		f(x^\ast) \leq f(x^\pi) &\leq&f(x) + L ||x^\pi - x||_1\\
		&=& f(x) + L {\rm dist}_\mathcal{C}(x)\\
		&=& f(x) + c L {\rm dist} _\Delta F(x)
	\end{eqnarray*}
	Hence, penalty function $P_\alpha$ is exact with $\bar{\alpha} = cL$.
\end{proof}
\begin{Remark}
	The significance of this result is that it will work even for those points where MPVC-MFCQ does not hold, so this result is stronger than \cite[Corollary 3.9]{hoheisel}.\\\\
 We illustrate this for the MPVC given in  Example \ref{exmpl2}, which is
	\begin{eqnarray*}
	\min ~x_1^2 + x_2^2 &&\\
	g(x) = x_1 &\leq& 0 \\
	H(x) = x_2 &\geq& 0\\
	G(x) H(x) = -x_1 x_2 &\leq& 0.	
	\end{eqnarray*}
	Then $x^\ast = (0,0)$ is global minimizer of this program. At $x^\ast$ MPVC-MFCQ and MPVC-GMFCQ fail to hold, but MPVC-generalized-pseudonormality holds, consequently MPVC-generalized-quasinormality holds.\\ \hspace*{0.5cm} Now, the penalized problem associated to above MPVC stated in Theorem \ref{thm-extpenl} is given as
	\begin{equation*}
	P_\alpha (x) ~=~ x_1^2 + x_2^2 + \alpha [\max \{0,g(x)\} + \max \{0,-H(x), \min \{G(x), H(x)\} \}]
	\end{equation*}
	also has global optimal solution at $x^\ast = (0,0)$ for all $\alpha \geqslant 0$. Hence, $P_\alpha(x)$ is exact at $x^\ast$. \\
\end{Remark}
\section{Relations among the various MPVC-CQs :}
\label{relcq}
This section is devoted to establish some possible relationships among the MPVC-CQs, which we have defined. Though, in section \ref{Pre},  Proposition \ref{implcq} shows  that MPVC-MFCQ implies other weaker CQs. But, it is not known how MPVC-ACQ is related with most of the former CQs in Proposition \ref{implcq}.  In previous section, we have shown that the MPVC-generalised quasinormality is the weakest condition for exactness of the penalty function. On the other hand, the MPVC-ACQ is not strong enough to guarantee the exact penalty results. It suggests that MPVC-ACQ must be weaker than others. Indeed, we show that the MPVC-generalised quasinormality is strictly stronger than MPVC-ACQ. \\\\
\hspace*{0.5 cm} We begin by considering the abstract form of MPVC (\ref{mpvc vector}), again as
\begin{equation}
\label{relcq1}
\min f(x)~~~~~~~{\rm s.t.}~~ ~~F(x)\in \Delta
\end{equation}
where $f$ is locally Lipschitz and $F$ is continuously differentiable.\\
\hspace*{0.5 cm} Now,  we consider the following class of associated perturbed problems
\begin{equation*}
\min f(x)~~~~~~~{\rm s.t.}~~ ~~F(x)~+~p \in \Delta
\end{equation*}
for some parameter $p \in \mathbb{R}^t, t=m+l+q$.\\ 
 \hspace*{0.5 cm}The feasible set of this perturbed problem can be define by means of the multifunction
\begin{equation}
\label{perturbation}
	M(p) := \{x \in \mathbb{R}^n ~|~ F(x)+ p \in \Delta \}
\end{equation}
usually called perturbation map. It is easy to see that $ \mathcal{C}= F^{-1} (\Delta)= M(0).$\\
\hspace*{0.5 cm} The applicability of calculus of multifunctions in optimization problems emerged the following notion of calmness for multifunction, from\cite{Rocka}.
\begin{Definition}
Let $\Phi : \mathbb{R}^p \rightrightarrows \mathbb{R}^q$ be a multifunction with a closed graph and $(u, v) \in gph \Phi$. Then we say that $\Phi$ is calm at $(u, v)$ if there exist neighbourhoods $U$ of $u$, $V$ of $v$ and a modulus $L \geq 0$ such that
\begin{equation}
\Phi (u') \cap V ~\subseteq~ \Phi(u) + L ||u - u'|| \mathbb{B}~~~~~~~~\forall~u' \in U
\end{equation}
where $\mathbb{B}:=\mathbb{B}(0,1)$.
\end{Definition}
The significance of the calmness stems in the following result, see \cite[Corollary 1]{henrion} or \cite{Pang}.
 \begin{Proposition}
 \label{calm-leb}
 Let $x^\ast \in M(0)$ be a feasible point for (\ref{relcq1}). Then the following are equivalent
 \begin{enumerate}
 \item $M$ is calm at $(0, x^\ast)\in {\rm gph} M.$
 \item Local error bounds exist \emph{i.e.} there exist constants $\delta>0$ and $c>0$ such that
 \begin{equation*}
		{\rm dist}_{F^{-1} (\Delta)}(x) \leqslant c~ {\rm dist}_\Delta (F(x))
	\end{equation*}
holds for all $x \in \mathbb{B}(x^\ast, \delta)$.
 \end{enumerate}
 \end{Proposition}
Now, we recall the GMFCQ from \cite[Definition 3.7]{hoheisel} and we show that in MPVC-setup, this definition is actually equivalent to MPVC-GMFCQ given in section \ref{Pre}.
\begin{Definition}
	\label{gmfcq 2}
Let $x^\ast$ be feasible for (\ref{mpvc vector}), then the generalized Mangasarian-Fromovitz constraint qualification (GMFCQ) holds at $x^\ast$ if the following holds
\begin{equation}
\label{gmfcq}	\begin{rcases}
		F'&(x^\ast)^T \lambda =~ 0 \\
		\lambda \in & N_\Delta (F(x^\ast))
	\end{rcases}
	\text{$\Rightarrow \lambda = 0 $}
	\end{equation}
\end{Definition}
Now, we show that the two definitions are equivalent. For this, we need the limiting normal cones of some relevant sets \cite[Lemma 3.2]{hoheisel 2}.
\begin{equation*}
N_\Omega (a,b) ~=~ \left\{ \left( \begin{array}{cc}
 \xi\\ \zeta
\end{array}\right) : \begin{array}{ccc}
\xi = ~0~ =\zeta &;& if ~ a >0,~b <0\\
\xi =0 ,~ \zeta \geqslant 0 &;& if ~a > 0, ~ b=0 \\
\zeta \geqslant 0, \xi  \cdot \zeta  = 0 &;&if ~a ~=~0~= ~b \\
\xi \leqslant 0 ,~ \zeta = 0 &;& if ~a = 0, ~b <0\\
\xi \in \mathbb{R} ,~ \zeta = 0 &;& if ~a = 0, ~b > 0
\end{array}
\right \}
\end{equation*}
\begin{equation*}
	N_{(-\infty, 0]} (a) ~=~ \left \{ \begin{array}{ccc}
	\{0\} &;& a < 0 \\
   \ [0, \infty) &;& a=0 \\
	\phi &;& a > 0
	\end{array} \right \}
\end{equation*}
\begin{equation*}
N_{\{0\}} (0) ~=~ \mathbb{R}
\end{equation*}
With the structure of above cones, we can establish the equivalence between Definitions \ref{gmfcq 1} and \ref{gmfcq 2}, as follows :
\begin{Lemma}
\label{gmfcq1-2}
Definition \ref{gmfcq 2} is equivalent to MPVC-GMFCQ.
\end{Lemma}
\begin{proof}
Firstly, we may write the limiting normal cone $N_\Delta (F(x^\ast))$ according to \cite[Proposition 6.41]{Rocka} as
\begin{equation*}
N_\Delta (F(x^\ast)) = \prod_{i=1}^{m} N_{(-\infty, 0]} (g_i(x^\ast)) \times \prod_{j=1}^{l} N_{\{0\}} (h_j(x^\ast)) \times \prod_{i=1}^{q} N_\Omega (G_i(x^\ast, H_i(x^\ast))
\end{equation*}
Hence, condition (\ref{gmfcq}) in Definition \ref{gmfcq 2} is equivalent to \\
\begin{equation*}
	 \sum_{i=1}^{m} \lambda_i \nabla g_i(x^\ast) + \sum_{j=1}^{l} \mu_j \nabla h_j (x^\ast) + \sum_{i=1}^{q} \eta^G_i \nabla G_i(x^\ast)
	-\sum_{i=1}^{q} \eta^H_i \nabla H_i(x^\ast)~=~0
\end{equation*}	
where\\
$	~~~~~~~~~\lambda_i \in N_{(-\infty, 0]} (g_i(x^\ast)) ~~~\forall~i = 1,...,m $\\
$~~~~~~~~~~ \mu_j \in N_{\{0\}}(h_j(x^\ast)) ~~~\forall ~j = 1,...,l$\\
$~~~~~~~~~~ (\eta_i^G, -\eta_i^H) \in -N_\Omega(G_i(x^\ast), H_i(x^\ast)) ~~~\forall~i=1,...,q$ \\
\begin{equation*}
\Longrightarrow (\lambda, \mu, \eta^G, \eta^H) ~=~0
\end{equation*}
which is the MPVC-GMFCQ.
\end{proof}

In \cite[Proposition 3.8]{hoheisel}, it has been given that MPVC-GMFCQ equivalently condition (\ref{gmfcq}) guarantees the calmness of $M$ at $(0, x^\ast) \in  gph M$ for any feasible point $x^\ast \in M(0)$ of  MPVC (\ref{mpvc vector}), and thus exactness of penalty function (\ref{penalty}) follows, see \cite[Corollary 3.9]{hoheisel}. Hence, Lemma \ref{gmfcq1-2} immediately improves the result \cite[Theorem  4.5]{hoheisel}.\\
\hspace*{0.5 cm} Now in order to derive the said relation, we need the tangent cone of set $\Delta$, which is hard to compute directly. Fortunately, we have the following result, which reduces the difficulty of such computation and will be used to derive the main Theorem of this section.
\begin{Lemma}
\label{tangent prod}
Let $x^\ast$ be feasible for MPVC, then the tangent cone is given by
\begin{equation*}
T_\Delta (F(x^\ast)) = \prod_{i=1}^{m} T_{(-\infty, 0]} (g_i(x^\ast)) \times \prod_{j=1}^{l} T_{\{0\}} (h_j(x^\ast)) \times \prod_{i=1}^{q} T_\Omega (G_i(x^\ast), H_i(x^\ast))
\end{equation*}
\end{Lemma}
\begin{proof}
	Here we need to show only ''$\supseteq$'' inclusion, another "$\subseteq$" follows from \cite[Proposition 6.41]{Rocka}. Choose arbitrary elements $d_{g_i} \in T_{(-\infty, 0]} (g_i(x^\ast)), ~ d_{h_j} \in T_{\{0\}} (h_j(x^\ast)) $ and $(d_{G_i}, d_{H_i}) \in T_\Omega (G_i(x^\ast), H_i(x^\ast))$, and define
	\begin{equation*}
		d := (d_{g_i,~i=1,...,m}, ~d_{h_j, ~j=1,...,l}, ~(d_{G_i}, d_{H_i})_{i=1,...,q})
	\end{equation*}
	Following the definition of a tangent vector, there exist sequences
	\begin{eqnarray}
\nonumber	d^k_{g_i} \rightarrow d_{g_i}, ~t^k_{g_i} \downarrow 0 &{\rm with}& g_i(x^\ast) + t^k_{g_i} d^k_{g_i} \leq 0 \\
\nonumber	d^k_{h_j} \rightarrow d_{h_j}, ~t^k_{h_j} \downarrow 0 &{\rm with}& h_j(x^\ast) + t^k_{h_j} d^k_{h_j} = 0 \\
\label{vc 1}
	(d^k_{G_i}, d^k_{H_i}) \rightarrow (d_{G_i}, d_{H_i}), ~t^k_{G_i H_i} \downarrow 0 &{\rm with}& (H_i(x^\ast) + t^k_{G_i H_i} d^k_{H_i}) \geq 0 \\
	\label{vc 2}
	&{\rm and}& (G_i(x^\ast) + t^k_{GH_i} d^k_{G_i}) (H_i(x^\ast) + t^k_{GH_i} d^k_{H_i}) \leq 0
	\end{eqnarray}
	$\forall ~k \in \mathbb{N}$. Consequently, we have
	\begin{equation*}
	d^k := \left (d^k_{g_i,~i=1,...,m}, ~d^k_{h_j, ~j=1,...,l}, ~(d^k_{G_i}, d^k_{H_i})_{i=1,...,q} \right ) ~\rightarrow~ d
	\end{equation*}
	Now to prove the required result we have to show that $d \in T_\Delta (F(x^\ast))$, that is we have to find a sequence $t^k \downarrow 0$ such that $F(x^\ast) + t^k d^k \in \Delta,~\forall ~k \in \mathbb{N}$. \\
	Define
	\begin{equation*}
		t^k := \min \{t^k_{g_i, i=1,...,m}, t^k_{h_j, j=1,...,l}, t^k_{G_i H_i, 1,...,q}\}
	\end{equation*}
	$\forall ~k \in \mathbb{N}$. Clearly $t^k \downarrow 0$, and it remains to show $F(x^\ast) + t^k d^k \in \Delta ~~\forall~ k \in \mathbb{N}$. Now choose $k \in \mathbb{N}$ arbitrarily but fixed, and recall that $x^\ast$ is feasible for MPVC. Then for every $i =1,...,m$, two cases can arise, either $d^k_{g_i} < 0$ or $d^k_{g_i} \geq 0$.\\
	 If $d^k_{g_i} < 0$, then we have
	 \begin{equation*}
	 g_i(x^\ast) + t^k d^k_{g_i} < g_i(x^\ast) \leq 0
	 \end{equation*}
	 and if $d^k_{g_i} \geq 0$, then
	 \begin{equation*}
	 g_i(x^\ast) + t^k d^k_{g_i}  \leq g_i(x^\ast) + t^k_{g_i} d^k_{g_i} \leq 0
	 \end{equation*}
	 Since $h_j(x^\ast) = 0$ and $t^k_{h_j} > 0, ~~\forall~j = 1,...,l$, therefore $d^k_{h_j} = 0$. Consequently, we have
	 \begin{equation*}
	 h_j(x^\ast) + t^k d^k_{h_j} = 0
	 \end{equation*}
	\textbf{Case (I) :} Consider $H_i(x^\ast) > 0$, then either $G_i(x^\ast) = 0$ or $G_i(x^\ast) < 0$ \\
	 \hspace*{0.5 cm} If $G_i(x^\ast) = 0$, that is $i \in I_{+0}$ then because of $d^k_{H_i} \rightarrow d_{H_i}$ and $t^k_{G_i H_i} \downarrow 0$, we have by eq. (\ref{vc 1})
	 \begin{equation}
	 \label{eq1}
	 H_i(x^\ast) + t^k_{G_i H_i} d^k_{H_i} > 0 ~~~~~;~~~~~~~~~~~~~~~~\forall~k \in \mathbb{N} ~~{\rm sufficiently ~large}
	 \end{equation}
	Then $H_i(x^\ast) + t^k d^k_{H_i} > 0$ also holds for sufficiently large $k \in \mathbb{N}$. Again (\ref{eq1}) yields with (\ref{vc 2})
	\begin{equation*}
	G_i(x^\ast) + t^k_{G_i H_i} d^k_{G_i} \leq 0 ~~~~~~;~~~~~~~~~~\forall~k \in \mathbb{N}
	\end{equation*}
	and hence $d^k_{G_i} \leq 0$. This implies
	\begin{eqnarray*}
			G_i(x^\ast) + t^k d^k_{G_i} &\leq &0 ~~~~~~;~~~~~~~~~~\forall~k \in \mathbb{N}~{\rm sufficiently ~large}\\
\Rightarrow ~	\left( H_i(x^\ast) + t^k d^k_{H_i} \right) \left( G_i(x^\ast) + t^k d^k_{G_i} \right)  &\leq& 0
	\end{eqnarray*}
	that is $F(x^\ast) + t^k d^k \in \Delta$ for all $k \in \mathbb{N}$.\\
	\hspace*{0.5 cm} If $G_i(x^\ast) < 0$, that is $i \in I_{+-}$ then $H_i(x^\ast) + t^k d^k_{H_i} > 0$ for all $k$ sufficiently large similarly as above, and also
	\begin{equation*}
	H_i(x^\ast) + t^k_{G_i H_i} d^k_{H_i} > 0
	\end{equation*}
	gives
	\begin{equation*}
		(G_i(x^\ast) + t^k_{G_i H_i} d^k_{H_i}) \leq 0 ~~~~~~~~~~~~~~~~~~{\rm by~eq~(\ref{vc 2})}
	\end{equation*}
	hence
	\begin{equation*}
	    ~~~~~~~~~~~~~~~~~~~~~(G_i(x^\ast) + t^k d^k_{H_i}) \leq 0 ~~~~~~~;~~~~~~~~~~~~{\rm for~all ~sufficiently ~large ~k}	
	\end{equation*}
	It again provides
	\begin{equation*}
	\left( H_i(x^\ast) + t^k d^k_{H_i} \right) \left( G_i(x^\ast) + t^k d^k_{G_i} \right)  \leq 0~~~~~~;~~~~~~~~~~~~~{\rm for~all ~sufficiently ~large ~k}
	\end{equation*}
	for all $i \in I_{+-}$, that is $F(x^\ast) + t^k d^k \in \Delta$ for all $k \in \mathbb{N}$ sufficiently large.\\
\textbf{Case (II) :} Now we consider $H_i(x^\ast) = 0$, then $d^k_{H_i} \geq 0$ and hence $H_i(x^\ast) + t^k d^k_{H_i} \geq 0$  for all $k \in \mathbb{N}$ and now we consider possibilities of $G_i(x^\ast)$ for both cases of $d^k_{H_i}$. \\
 (i) Suppose $d^k_{H_i} > 0$ firstly, then we have \\
 \begin{equation*}
 	G_i(x^\ast) + t^k_{G_i H_i} d^k_{G_i} \leq 0 ~~~~~~~~~~~;~~~~~~~~~~~~~\forall~i \in I_{0+} \cup I_{0-} \cup I_{00}
 \end{equation*}
 this gives
 \begin{equation*}
 ~~~~~~~~~~~~~~~~~~~~~~~~~~~~~~~~~~~~G_i(x^\ast) + t^k d^k_{G_i} \leq 0 ~~~~~~~~; ~~~~~~~~~~~~~{\rm for ~ sufficiently ~large } ~k \in \mathbb{N}
 \end{equation*}
 and hence
 \begin{equation*}
 \left( H_i(x^\ast) + t^k d^k_{H_i} \right) \left( G_i(x^\ast) + t^k d^k_{G_i} \right)  \leq 0
 \end{equation*}
 for all $i \in I_{0+} \cup I_{0-} \cup I_{00}$ and result holds.\\
 (ii) Now suppose $d^k_{H_i} = 0$ then $H_i(x^\ast) + t^k d^k_{H_i} = 0$ and hence
 \begin{equation*}
 \left( H_i(x^\ast) + t^k d^k_{H_i} \right) \left( G_i(x^\ast) + t^k d^k_{G_i} \right)  = 0
 \end{equation*}
 for all $i \in I_{0+} \cup I_{0-} \cup I_{00}$ and it obviously produce result as $F(x^\ast) + t^k d^k \in \Delta$ for all $k \in \mathbb{N}$ sufficiently large.
\end{proof}
\hspace*{0.5 cm} Here is the main result of this section, which states that MPVC-ACQ is weaker than MPVC-generalized-quasinormality.
\begin{Theorem}
	\label{relation}
Let $x^\ast$ be feasible for MPVC such that MPVC-generalized-quasinormality holds at $x^\ast$. Then MPVC-ACQ also holds at $x^\ast$.
\end{Theorem}
\begin{proof}The Lemma \ref{error bound}  shows that MPVC-generalized-quasinormality yields the existence of local error bounds and by Proposition \ref{calm-leb} this is equivalent to calmness of the perturbation map $M(p)$ at $(0, x^\ast)$. Since $F$ is continuously differentiable, hence locally Lipschitz, therefore from \cite[Proposition 1]{henrion}, we obtain
\begin{equation*}
	T_\mathcal{C}(x^\ast) ~=~L_\mathcal{C}(x^\ast)
\end{equation*}
where $L_\mathcal{C}(x^\ast)$ is the linearized cone of feasible region $\mathcal{C}$ at $x^\ast$ and is defined as
\begin{equation*}
L_\mathcal{C}(x^\ast) = \{d \in \mathbb{R}^n  ~|~ \nabla F(x^\ast)^T d \in T_\Delta (F(x^\ast)) \}
\end{equation*}
Since we have by Lemma \ref{tangent prod}
\begin{equation*}
T_\Delta (F(x^\ast)) = \prod_{i=1}^{m} T_{(-\infty, 0]} (g_i(x^\ast)) \times \prod_{j=1}^{l} T_{\{0\}} (h_j(x^\ast)) \times \prod_{i=1}^{q} T_\Omega (G_i(x^\ast, H_i(x^\ast)).
\end{equation*}
Therefore, $L_\mathcal{C}(x^\ast)$ can be written as
\begin{eqnarray*}
L_\mathcal{C}(x^\ast) &=& \{ d \in \mathbb{R}^n ~|~ \nabla g_i(x^\ast)^T d \in T_{(-\infty, 0]}(g_i(x^\ast)) ~~\forall~i = 1,...,m,\\
&&~~~~~~~~~~~~~~\nabla h_j(x^\ast)^T d \in T_{\{0\}}(h_j(x^\ast)) ~~~~~\forall~j = 1,...,l,\\
&&~~~~~~~~~~~~~~(\nabla G_i(x^\ast)^T d, \nabla H_i(x^\ast)^T d) \in T_\Omega(G_i(x^\ast), H_i(x^\ast))~~\forall~i = 1,...,q \} \\
&&\\
&=& \{ d \in \mathbb{R}^n ~|~ \nabla g_i(x^\ast)^T d \leq0 ~~~~\forall ~i \in I_g \\
&&~~~~~~~~~~~~~~ \nabla h_j(x^\ast)^T d = 0 ~~~~\forall ~j = 1,...,l \\
&&~~~~~~~~~~~~~~ \nabla H_i(x^\ast)^T d = 0 ~~~~\forall ~i \in I_{0+} \\
&&~~~~~~~~~~~~~~ \nabla H_i(x^\ast)^T d \geq 0 ~~~~\forall ~i \in I_{00} \cup I_{0-} \\
&&~~~~~~~~~~~~~~ \nabla G_i(x^\ast)^T d \leq 0 ~~~~\forall ~i \in I_{+0} \}\\
&&\\
&=& L_{MPVC}(x^\ast)
\end{eqnarray*}
Here $L_{MPVC}$ is the linearized cone of MPVC as defined in Definition \ref{acq}, and consequently we have  $T_\mathcal{C}(x^\ast) ~=~ L_\mathcal{C}(x^\ast) ~=~L_{MPVC}(x^\ast) $, that is MPVC-ACQ is satisfied at $x^\ast$.
\end{proof}
\begin{Remark}
MPVC-ACQ is strictly weaker than MPVC-generalized-quasinormality, we illustrate it as follows.
\end{Remark}
\begin{Example} We consider the MPVC
\begin{eqnarray*}
	\min f(x) = |x_1|+|x_2|\\
	g(x) = x_1 + x_2 &\leq& 0 \\
	H(x) = x_1 &\geq& 0 \\
	G(x) H(x) = x_1 (x_1^2 - x_2^2) &\leq& 0
\end{eqnarray*}
The point $x^\ast = (0,0)$ is feasible and all constraints are active at $x^\ast$. For this program MPVC-generalized-quasinormality and all stronger CQs fail to hold at $x^\ast$, but MPVC-ACQ holds because $T_\mathcal{C}(x^\ast)~=~\mathcal{C}~=~L_{MPVC}(x^\ast)$ for $\mathcal{C}$ being the feasible region for the program.
\end{Example}
\begin{Remark}
In the above example, it is easy to see that $P_\alpha(x)$ is exact at $x^\ast=(0,0)$ but MPVC-generalized-quasinormality is violated at $x^\ast$. Hence, in general,  converse of the Theorem \ref{thm-extpenl} is not true.
\end{Remark}
\hspace*{0.5 cm} Finally, we have shown that the following implications hold for a local minimum $x^\ast$ of MPVC given in (\ref{initprblm}).
\begin{eqnarray*}
 &MPVC-MFCQ &\\
&\Downarrow &\\
& MPVC-GMFCQ &\\
&\Downarrow &\\
&MPVC-generalized~pseudonormality &\\
&\Downarrow &\\
&MPVC-generalized~quasinormality & \\
&\Downarrow &\\
&  MPVC-ACQ \Longleftarrow Calmness~ of~ M(p)~ at ~(0, x^\ast) \Longrightarrow exactness~of~P_\alpha &
\end{eqnarray*}	
\section{Concluding Remarks}
\label{Conc}
We have used a local error bound result from \cite{khare} to establish an exact penalty result for MPVC- tailored penalty function  $P_\alpha$ under a very weak and new assumption, the MPVC-generalized quasinormality. This CQ turns out to be strictly stronger than MPVC-ACQ, and has been illustrated  by an example. We conclude this paper having a challenge of investigating reasonable weak conditions for exactness of classical $l_1$-penalty function for MPVC.

\end{document}